\documentclass[12pt,reqno]{amsart}   	% use "amsart" instead of "article" for AMSLaTeX format
\usepackage[letterpaper, margin=1in]{geometry}
\usepackage{amssymb}
\usepackage{amsmath}
\usepackage{amsfonts}
\usepackage[utf8]{inputenc}
\usepackage{amsthm}
\usepackage{tikz-cd}
\usepackage{array}

\newtheorem{thm}{Theorem}[section]
\newtheorem{lemma}[thm]{Lemma}
\newtheorem{df}[thm]{Definition}
\newtheorem{prop}[thm]{Proposition}

\newtheorem{cor}[thm]{Corollary}

\newtheorem{prob}[thm]{Problem}

\newtheorem{conj}[thm]{Conjecture}

\newtheoremstyle{remark}
    {\dimexpr\topsep/2\relax} % space above
    {\dimexpr\topsep/2\relax} % space below
    {}          % body font
    {}          % indent amount
    {\bfseries} % theorem head font
    {.}         % punctuation after theorem head
    {.5em}      % space after theorem head
    {}          % theorem hed spec. (empty = "normal")

\theoremstyle{remark}
\newtheorem{remark}[thm]{Remark}
\newcommand{\RR}{\mathcal{R}}
\newcommand{\Z}{\mathbb{Z}}

\newcommand{\R}{\mathbb{R}}

\newcommand{\N}{\mathbb{N}}
\newcommand{\OO}{\mathcal{O}}

\newcommand{\brac}[1]{\left ( #1 \right )}

\title{Sums of squares on curves and surfaces}
\author{Bartłomiej Bychawski, Bartosz Głowacki, Tomasz  Kowalczyk}
\date{}

\begin{document}

\maketitle

\begin{abstract}
We study sums of higher even powers in the coordinate rings of singular planar curves $x^M=y^m$ for coprime positive integers $m<M$. We then show that the $2s$-Pythagoras number of real algebras of the form $\R[x,y,\sqrt{f_1},\sqrt{f_2},\dots, \sqrt{f_n}]$ are infinite, under some mild assumptions on the polynomials $f_1,f_2,\dots, f_n \in \R[x,y]$. We prove that all of the higher even Pythagoras numbers are finite for the ring of $0$-regulous functions on a $0$-regulous variety. We then show that the codimension of the bad set of order $2n$, for $n>1$, can be of codimension $2$, contrary to the quadratic case. 
\end{abstract}

\section*{Introduction}
In this paper we are interested in the problem of determining whether the (higher) Pythagoras numbers are finite or not for some functions rings arising in real algebraic geometry.
 Let us start with a definition.
 \begin{df}
For a commutative ring $R$ with identity and a positive integer $n>0$ its $2n$-Pythagoras number, denoted by $p_{2n}(R)$, is the smallest positive integer $g$ such that any element which is a sum of $2n$th powers can be written as a sum of at most $g$ $2n$th powers. If such number does not exist, we put $p_{2n}(R)=\infty$.
\end{df}
Here, we are interested only in sums of even powers. This is because, if we consider sums of $n$th powers for $n$ odd, then $-1$ is an $n$th power, and the $n$-Pythagoras number is finite, provided that $n!$ is invertible in $R$ (cf. \cite[Introduction]{km2024}).

Quite a lot is known about the $2$-Pythagoras number. Problem of determining finiteness of those invariants is particularly interesting in the case of a real rings (i.e. $-1$ is not a sum of squares). Any affine $\R$-algebra of dimension $1$ has finite $2$-Pythagoras number \cite{cldr1982}. On the other hand, any excellent ring of real dimension at least 3 has infinite $2$-Pythagoras number \cite{frs2004}. Dimension $2$ is least understood. We have $p_{2n}(\R[x,y])=\infty$ \cite{cldr1982, kv2023} as well as $p_{2n}(\R[x,y,\sqrt{f(x,y)}])=\infty$ for polynomials $f(x,y)$ satisfying some mild conditions (see \cite{bk2024, kv2023}). Surprisingly, it was recently proved that $p_2(\OO(\R^2))=4$ (see \cite{kowalczyk2025}), where $\OO(\R^2)$ is the ring of regular functions on the real plane. Let us also mention several recent papers on this topic \cite{benoist2025, bds2025, dghmy2026, dombek2025, hjp2025, krs2022, santiago2025,  tinkova2025}.

Higher Pythagoras number are much less understood. It is known that the theory of higher degree forms is much more complicated, than that of quadratic forms. Let us mention few papers on this topic  \cite{bk2023, becker1982, bp1996, grimm2015, km2024, kv2023}

The content of the paper is as follows. Section 1 contains the treatment of higher Pythagoras numbers of the coordinate rings of singular planar curves of the form $\{x^M=y^m \} \subset \R^2$ for coprime positive integers $m,M$. Next Section is devoted to the study of the $2s$-Pythagoras number of $2$-dimensional $\R$-algebras of the form $\R[x,y,\sqrt{f_1},\sqrt{f_2},\dots, \sqrt{f_n}]$ which correspond to surfaces in $\R^N$. This is a generalization of the results from \cite{bk2024}, where only coordinate rings of hypersurfaces in $\R^3$ were considered. Main aim of Section 3 is to show, that the higher even Pythagoras numbers of the ring of $0$-regulous functions on a $0$-regulous variety is finite. This provides a first family of rings -- other than fields and valuation rings -- with finite higher even Pythagoras numbers. In Section 4 we study the bad set of higher order. To be precise, for a given polynomial, we study the common zero set of all possible denominators in all possible representations of $f$ as a sum of $2n$th powers of rational functions. We show that if $n>1$ then the bad set can be of codimension 2, contrary to quadratic case. Last section contains some open problems.

\section{Sums of even powers on singular curves}

In this section we will be interested in sums of even powers in the coordinate rings of singular planar curves. Our motivation is the \cite[Note added in proof]{cldr1982}. It is said there that the rings of the form $\R[t^a,t^{2a-1}]$ have unbounded Pythagoras numbers. The proof of this fact does not seem to be published. Main result of this section is generalization of the aforementioned observation to higher even powers. We start with some definitions.

Let $X\subset \R^d$ be an algebraic subset and let $\mathcal{I}(X)$ be the ideal of all polynomials vanishing on $X$.
\begin{df}
    We define the coordinate ring of $X$, $\R[X]$, as the quotient
    $$ \R[x_1,x_2\dots, x_d] / \mathcal{I}(X)$$
\end{df}
We will be mostly interested in coordinate rings of curves. Let $m<M$ be two coprime positive integers. Define $X_{m,M}:=\{x^M=y^m \}$. It is an irreducible, singular curve, and so its coordinate ring is an  integral domain. We will study sums of even powers in the ring $\R[X_{m,M}]$.

To state results about sums of squares more easily, let us define the notion of length.

\begin{df}
    Let $R$ be a commutative ring with identity, $a \in R$ and $n>0$ be a positive integer. We define the $2n$-length of $a$, $\ell_{2n,R}(a)$ as the smallest number of $2n$th powers needed to represent $a$. If such number does not exist, we define $\ell_{2n,R}(a)=\infty$. Since element $a$ can be usually considered in different rings, we put the $R$ to stress out in which ring length is considered.
\end{df}

The main result of this section is the following theorem.

\begin{thm} \label{THM: podchodzenie krzywymi do płaszczyzny dla p_2n}
    Let     $  {
   \big\{
    (m_i,M_i)
    \big\}
    }_{i \in \N}    $
     be a sequence of coprime tuples of positive integers such that for every $i$, $m_i < M_i$ and  $\lim_{i \rightarrow \infty} m_i = \infty$. Then      $$
    \lim_{i \rightarrow \infty} p_{2n}
        \big(
        \R [X_{m_i,M_i}]
        \big)
    = \infty
    $$
\end{thm}

%What is both important and interesting is the fact that it gives us many families of curves sich that values of $p_{2n}$ of their associated rings converge to $p_{2n}\big(\R[x,y]\big) = \infty$.
%\textcolor{red}{BYCHU: dodac notke do ciagu $x^{2n-1} - y^n$ od ktorego się to zaczelo}

%Before we prove this theorem we have to prepare ourselves with two results.
Before we start the proof, we need some preliminary results.  Following fact was proven for $n=1$ in \cite[Theorem 4.1]{cldr1982} and for $n>1$ in \cite[Theorem 2.2]{kv2023}.
%\textcolor{red}{BYCHU: dodać cytowanie twierdzenia poniżej!!!}
\begin{thm} \label{THM: p_2n dla R^2}
    For any positive integer $n\geq1$ the equality $p_{2n}\big(\R[x,y]\big) = \infty$ holds.
\end{thm}

Let us prove a crucial lemma.
\begin{lemma} \label{LEM: zachowywanie długości przy przejściu z R^2 na R[X_m,M]}
Let $m<M$ be a pair of coprime positive integers and $X_{m,M}$ be the corresponding curve.  Let  $$\varphi: \R[x,y] \twoheadrightarrow \R[X_{m,M}]$$ be the natural projection. If the degree of $f$ is strictly smaller than $m$, then
    $$
    \ell_{2n, \R[x,y]} (f) = \ell_{2n,\R[X_{m,M}]}\big(\varphi(f)\big).
    $$
   
\end{lemma}
\begin{proof}
    Since $\varphi$ is a surjection, inequality $\ell_{2n, \R[x,y]} (f) \geq \ell_{2n,\R[X_{m,M}]}\big(\varphi(f)\big)$ holds. To finish the proof it is therefore enough to show that opposite inequality holds. We will restrict ourselves to the nontrivial case, that is $\ell_{2n,\R[X_{m,M}]}\big(\varphi(f)\big) < \infty$.

First of all, the coordinate ring of $X_{m,M}$ is easily seen to be isomorphic to the ring $\R[t^m,t^M]$. To stress out that we work in this particular ring, rather than in $\R[X_{m,M}]$, we denote by $\widetilde{\varphi}$ the epimorphism $\widetilde{\varphi}: \R[x,y] \rightarrow \R[t^m,t^M]$, and so, we will consider $\widetilde{\varphi}(f)$ instead of $\varphi(f)$.

% https://q.uiver.app/#q=WzAsMyxbMCwwLCJcXFJbeCx5XSJdLFsxLDAsIlxcUlt0Xm0sdF5NXSJdLFsxLDEsIlxcUltYX3ttLE19XSJdLFswLDEsIlxcd2lkZXRpbGRle1xcdmFycGhpfSJdLFswLDIsIlxcdmFycGhpIiwyXSxbMiwxLCJcXGlvdGEiLDJdXQ==
Let $f \in \R[x,y]$ be such that $\deg f <m$. Since $\deg\big(\widetilde{\varphi}(x)\big), \deg\big(\widetilde{\varphi}(y)\big) \leq M$, we obtain the bound 
    $$
    \deg\big(
    \widetilde{\varphi}(f)
    \big) \leq M \deg(f) < m\cdot M.
    $$
    Let us denote $\ell := \ell_{2n,\R[t^m,t^M]}
    \big(
    \widetilde{\varphi}(f)
    \big)$ and let $g_1,\ldots,g_\ell \in \R[t^m,t^M]$ be such that
    $$
     \widetilde{\varphi}(f) = \sum_{i=1}^\ell g_i^{2n}.
    $$
    We can immediately observe that for any $i =1,2,\dots,\ell$ inequality $2n \cdot \deg(g_i) \leq \deg\big(
    \widetilde{\varphi}(f)
    \big) < m M$ holds, hence $\deg(g_i) < \frac{mM}{2n}$. 
    
    To finish the proof we will define a $\R$-linear map $\psi: \R[t^m,t^M] \rightarrow \R[x,y]$ which satisfies $\widetilde{\varphi} \circ \psi = \mathrm{id}_{\R[t^m,t^M]}$. If $t^k \in \R[t^m,t^M]$, then there exists $a_k,b_k \in \Z_{\geq 0}$ satisfying $k = a_k m + b_k M$. Let us therefore define $\psi(t^k) := x^{a_k} y^{b_k}$ and extend linearly. By construction,  $$\widetilde{\varphi} \circ \psi(t^k) = t^{a_k m + b_k M} = t^k$$ hence indeed equality $\widetilde{\varphi} \circ \psi = \mathrm{id}_{\R[t^m,t^M]}$ holds. Additionally, we have the following chain of inequalities
    $$
    \deg \psi(t^k) = a_k + b_k = \frac{1}{m} \big( a_k m + b_k m\big) \leq \frac{1}{m} \big(a_km +  b_kM\big) = \frac{k}{m} = \frac{1}{m} \cdot \deg(t^k).
    $$
    As a consequence, the inequality $\deg \psi(h) \leq \frac{1}{m} \cdot \deg(h)$ holds for any $h \in \R[t^m,t^M]$.

    Now let $\mathfrak{s} := \sum_{i=1}^{\ell} {\psi(g_i)}^{2n}$ and observe that the following inequalities holds
    $$
    \deg(\mathfrak{s}) \leq \max_{i =1,\dots,\ell} \Big\{ 
    2n \cdot \deg \psi(g_i) 
    \Big\} 
    \leq
    \max_{i =1,\dots,\ell} \bigg\{ 
    \frac{2n}{m} \cdot \deg(g_i) 
    \bigg\}
    <
    \max_{i =1,\dots,\ell} \bigg\{ 
    \frac{2n}{m} \cdot \frac{mM}{2n} 
    \bigg\} = M.
    $$
    Our aim now, is to show that $\mathfrak{s} = f$. To accomplish this, let $\varepsilon := f -\mathfrak{s}$. We have
    \begin{multline*}
    \widetilde{\varphi} (\varepsilon) = \widetilde{\varphi}(f) - \widetilde{\varphi}(\mathfrak{s}) = 
    \sum_{i=1}^\ell g_i^{2n} - 
    \widetilde{\varphi}
    \Bigg(
    \sum_{i=1}^{\ell} {\psi(g_i)}^{2n}
    \Bigg)
    =
    \\
    =
    \sum_{i=1}^\ell g_i^{2n} - 
    \sum_{i=1}^{\ell} 
    {
    \Big(
    (\widetilde{\varphi} \circ \psi)(g_i)
    \Big)}^{2n} = 
    \sum_{i=1}^\ell g_i^{2n} - \sum_{i=1}^\ell g_i^{2n} = 0,
    \end{multline*}
    hence $\varepsilon \in \ker (\widetilde{\varphi})$. However, it is a straightforward calculation that $\ker (\widetilde{\varphi}) = (x^M - y^m)$, hence $\deg(\varepsilon) \geq \deg(x^M - y^m) = M$ or $\varepsilon = 0$. Since $\deg(f) < M$ and $\deg(\mathfrak{s}) < M$ we obtain $\deg(\varepsilon) \leq \max \big\{ \deg(f), \deg(\mathfrak{s}) \big\} < M$ which implies $\varepsilon = 0$. This finishes the proof, since we obtain
    $$
    f = \mathfrak{s} = \sum_{i=1}^{\ell} {\psi(g_i)}^{2n},
    $$
    which shows that $\ell_{2n,\R[x,y]}(f) \leq \ell = \ell_{2n,\R[X_{m,M}]}
    \big(
    \widetilde{\varphi}(f)
    \big)$ as desired.
\end{proof}

\noindent
\textit{Proof of Theorem \ref{THM: podchodzenie krzywymi do płaszczyzny dla p_2n}}
%\textcolor{red}{BYCHU: napisać dowód}
It is enough to show that 
$
    \liminf_{i \rightarrow \infty} p_{2n}
        \big(
        \R [X_{m_i,M_i}]
        \big)
    \geq k
$
for all positive integers $k$. Let $k$ be an arbitrary positive integer. Since $p_{2n}\big(\R[x,y]\big) = \infty$, there exists a polynomial $f \in \R[x,y]$ such that $\ell_{2n, \R[x,y]}(f) \geq k$. Let us put $d := \deg(f)$. Since $\lim_{i \rightarrow \infty} m_i = \infty$, there exists $I \in \N$ such that $m_i > d$ for all $i \geq I$. Therefore for any $i \geq I$ we can now apply Lemma \ref{LEM: zachowywanie długości przy przejściu z R^2 na R[X_m,M]} to obtain 
$$
   k \leq \ell_{2n, \R[x,y]} (f) = \ell_{2n,\R[X_{m_i,M_i}]}\big(\varphi_{(m_i,M_i)}(f)\big),
$$
where $\varphi_{(m_i,M_i)}$ denotes the natural projection from $\R[x,y]$ onto $\R[X_{m_i,M_i}]$.

Consequently,  $p_{2n}\big(\R[X_{m_i,M_i}]\big) \geq k$ for all $i \geq I$, hence inequality
$$
    \liminf_{i \rightarrow \infty} p_{2n}
        \big(
        \R [X_{m_i,M_i}]
        \big)
    \geq k
$$
indeed holds for arbitrary $k \in \N$ as desired.
\hfill \qed

\begin{remark}
    Note that if $n=1$ then we know that $p_2(\R[X_{m,M}]) < \infty$ (see \cite[Theorem 2.7]{cldr1982}). However, if $n>1$ then it is not known if those number are finite. It is not even known if $p_{2n}(\R[x])$ is finite or not (cf. \cite[Question 5.1]{kv2023}).
\end{remark}

\begin{remark}
Combining Theorem \ref{THM: podchodzenie krzywymi do płaszczyzny dla p_2n} with the results from \cite{scheiderer2017} we obtain the following ($m<M$) bounds 
$$ \frac{m}{2} \leq p_2(\R[X_{m,M}]) \leq 2m $$
(see \cite[Theorem 3.6]{scheiderer2017} and  \cite[Theorem 3.1]{cldr1982}).

\end{remark}

\section{Sums of squares on surfaces}
In this section we record some generalizations of the results in \cite{bk2024} to rings obtained by adjoining square roots of several polynomials. Those rings correspond to algebraic surfaces in $\R^N$. In \cite{bk2024} only coordinate rings of surfaces in $\R^3$ were considered.

\begin{df}
    Given a finite sequence $f_1,\dots, f_n\in\R[x,y]$ and a subset $S\subseteq\{1,\dots, n\}=[n]$ we will write
    \[
    f_S = \prod_{i\in S} f_i.
    \]
    We will call such a sequence square-independent if none of the non-empty products $f_S$ is a square in $\R[x,y]$.
\end{df}

\begin{df}
    Let $f\in\R[x,y]$ be a polynomial and $d = \deg_y f$. Let $\alpha x^b y^d$ be a monomial with the largest possible $b$ and nonzero coefficient $\alpha$ among all such monomials in $f(x,y)$. Then we will call $\alpha$ the $y$-leading coefficient of $f$.
\end{df}

\begin{thm}\label{positive-leading}
    Let $f_1,\dots, f_n\in \R[x,y]$ be square-independent polynomials. Assume that they have a common zero and their $y$-leading coefficients are positive. Then
    \[
    p_2\brac{\R[x,y,\sqrt{f_1},\dots, \sqrt{f_n}]} = \infty.
    \]
\end{thm}

\begin{proof}
    Consider a sequence of positive integers $(r_m)_{m\geq1}$ such that
    \[
    r_1>1+\max_i \deg f_i\quad \text{and} \quad r_k>\sum_{i=1}^{k-1} r_i. 
    \]
    We now define a sequence of polynomials $(F_m)_{m\geq 1}$ by letting $F_1=1$ and
    \[
    F_{n+1} = F_n\brac{y-x^{r_n}}^2+1.
    \]

    Assume to the contrary that the 2-Pythagoras number of $\R[x,y,\sqrt{f_1},\dots, \sqrt{f_n}]$ is finite and denote its value by $L$. Hence we have
    \[
    F_{L+1} = \sum_{i=1}^L h_i^2,
    \]
    for some $h_i\in \R[x,y,\sqrt{f_1},\dots, \sqrt{f_n}]$. Write
    \[
    h_i=\sum_{S\subseteq[n]}a_{i,S}\sqrt{f_S},
    \]
where $f_\varnothing = 1$.
    
    This yields
    \[
    F_{L+1} = \sum_{\nu=1}^L \sum_{S\subseteq[n]}f_Sa_{\nu,S}^2 = \sum_{\nu=1}^L a_{\nu,\varnothing}^2
        +
        \sum_{\nu=1}^L
        \sum_{\varnothing\neq S\subseteq[n]}
        f_Sa_{\nu,S}^2 .
    \]
    Substituting $y=x^{r_L}$ gives
    \[
     1 =\sum_{\nu=1}^L a_{\nu,\varnothing}(x,x^{r_L})^2 + \sum_{\nu=1}^L \sum_{\varnothing\neq S\subseteq[n]} f_S(x,x^{r_L}) a_{\nu,S}(x,x^{r_L})^2 .
    \]
    Since every nonzero term on the right hand side has a positive leading coefficient we obtain
    \[
        a_{\nu,S}(x,x^{r_L})=0 \qquad(\varnothing\neq S\subseteq[n]),
    \]
    and
    \[
            a_{\nu,\varnothing}(x,x^{r_L})=c_\nu\in\mathbb R,\qquad \sum_{\nu=1}^L c_\nu^2 =1 .
    \]
    After an orthogonal transformation (see \cite[Theorem 8.1.2]{pd2001}) we may thus assume
    \[
        a_{L,\varnothing}=1+(y-x^{r_L}) b_{L,\varnothing}
    \]
    and
    \[
        a_{\nu,\varnothing}=(y-x^{r_L}) b_{\nu,\varnothing} \qquad(1\leq \nu\leq L-1).
    \]
    Substituting these expressions into our equation, we get
    \[
    \begin{aligned}
            F_L(y-x^{r_L})^2+1
            &=
            1 + 2(y-x^{r_L}) b_{L,\varnothing} + (y-x^{r_L})^2 b_{L,\varnothing}^2     \\
            & \quad
            +(y-x^{r_L})^2\sum_{\nu=1}^{L-1} b_{\nu,\varnothing}^2+(y-x^{r_L})^2\sum_{\nu=1}^L \sum_{\varnothing\neq S\subseteq[n]} f_Sb_{\nu,S}^2 .
    \end{aligned}
    \]
    After cancelling the $1$'s, we see that $(y-x^{r_L})$ divides $b_{L, \varnothing}$, so we can write $b_{L, \varnothing} = t(y-x^{r_L})$. Let $D = 1+\max\deg f_i$. Consider the degree defined on monomials as
    \[
    \deg_D(x^ay^b) = a+Db,
    \]
    and break ties by exponent of $x$. Further extend it onto all polynomials in $\R[x,y]$ in the natural way. Since $r_m>D$, we see that \allowbreak\mbox{${\deg_D(y-x^{r_m}) = r_m}$}. On the other hand $D>\max\deg f_i$, so the coefficient next to the highest $\deg_D$ term of $f_i$ is its $y$-leading coefficient. Cancelling $(y-x^{r_L})^2$ from our equation we obtain
    \begin{equation}
        F_L = 2t + \brac{y-x^{r_L}}^2t^2 + \sum_{\nu=1}^{L-1} b_{\nu,\varnothing}^2
        + \sum_{\nu=1}^L \sum_{\varnothing\neq S\subseteq[n]} f_Sb_{\nu,S}^2 .
        \label{deg-comparison}
    \end{equation}
    Every nonzero square has positive $\deg_D$ leading coefficient, and by our previous observation every nonzero term $f_Sb_{\nu,S}^2$ also has positive $\deg_D$ leading coefficient. If $t\neq0$ then we see that $\deg_D$ of the right hand side is at least $2r_L$. On the other hand the left hand side has $\deg_D$ smaller than that by construction. Thus $t=0$.

    Equation \eqref{deg-comparison} now becomes
    \[
        F_L = \sum_{\nu=1}^{L-1} b_{\nu,\varnothing}^2
        + \sum_{\nu=1}^L \sum_{\varnothing\neq S\subseteq[n]} f_Sb_{\nu,S}^2 .
    \]
    We have reduced the number of unweighted square terms by one. Repeating the same argument with $F_L,\dots,F_2$ we obtain an identity
    \[
    F_1 = \sum_\lambda f_{S_\lambda}q_\lambda^2,
    \]
    However $F_1=1$, while evaluating at the common zero of $f_i$ the right hand side vanishes. This finishes the proof.
\end{proof}

\begin{cor}
    Let $f_1,\dots, f_n\in\R[x,y]$ be square-independent polynomials. Assume that they have a common zero and write $f_i = h_i+g_i$, where $h_i$ is the homogeneous part of $f_i$ of degree $\deg f_i$. Suppose there exists $v\in\R^2\setminus\{0\}$ such that $h_i(v)>0$ for every $i$. Then
    \[
    p_2\brac{\R[x,y,\sqrt{f_1},\dots,\sqrt{f_n}]} = \infty.
    \]
\end{cor}
\begin{proof}
     Choose $M\in\mathrm{GL}_2(\mathbb R)$ such that $M(0,1)=v$, and put
    \[
            \widetilde f_i=f_i\circ M.
    \]
    The highest homogeneous part of $\widetilde f_i$ is $h_i\circ M$, and
    \[
            (h_i\circ M)(0,1)=h_i(v)>0.
    \]
    Thus the coefficient of $y^{\deg f_i}$ in $\widetilde f_i$ is positive.
    In particular, each $\widetilde f_i$ has positive $y$-leading coefficient.
    Since $M$ induces an isomorphism of the corresponding rings, the
    claim follows from Theorem~\ref{positive-leading} applied to
    $\widetilde f_1,\ldots,\widetilde f_n$.
\end{proof}

\begin{thm}
    Let $f_1,\dots, f_n\in\R[x,y]$ be square-independent polynomials such that each $f_i$ is a sum of squares in $\R[x,y]$. Then
    \[
    p_2\brac{\R[x,y,\sqrt{f_1},\dots,\sqrt{f_n}]} = \infty.
    \]
\end{thm}
\begin{proof}
    Since each $f_i$ is a sum of squares, every $f_S$ is also a sum of squares. Indeed, if
    \[
            f_i=\sum_{j=1}^{N_i} g_{i,j}^2,
    \]
    then a product $f_S=\prod_{i\in S}f_i$ can be expanded as
    \[
            f_S = \sum_{(j_i)_{i\in S}} \left(\prod_{i\in S}g_{i,j_i}\right)^2.
    \]
    Thus each $f_S$ has finite length in $\R[x,y]$ and denote it by $\ell_S$. Set
    \[
    M = \sum_{S\subseteq [n]} \ell_S.
    \]
    Assume to the contrary, that $p_2(\R[x,y,\sqrt{f_1},\dots,\sqrt{f_n}])=L$. Since $p_2(\R[x,y]) = \infty$ we can choose a sequence $(G_m)_{m\geq1}$ of polynomials such that the length of $G_m$ is $m$ in $\R[x,y]$. We may write
    \[
    G_m = \sum_{\nu=1}^L h_\nu^2,
    \]
    for some $h_\nu\in\R[x,y,\sqrt{f_1},\dots,\sqrt{f_n}]$. Expanding and taking the $\R[x,y]$ component yields
    \begin{equation}
        G_m = \sum_{\nu=1}^{L} \sum_{S\subseteq [n]} f_Sa_{\nu,S}^2.
        \label{Gm-squares}
    \end{equation}
    Now for each $S\subseteq\{1,\dots, n\}$ choose a representation of $f_S$ as a sum of $\ell_S$ squares
    \[
    f_S=\sum_{j=1}^{\ell_S} u_{S,j}^2.
    \]
    Then
    \[
    f_Sa_{\nu,S}^2 = \sum_{j=1}^{\ell_S} (u_{S,j}a_{\nu,S})^2.
    \]
    Therefore the right hand side of \eqref{Gm-squares} contains at most $L\sum\ell_S = LM$ squares. Hence, by taking $m>LM$ we obtain a contradiction.
\end{proof}

\begin{prop}
    Let $f_1,\dots, f_n\in\R[x,y]$ be square-independent polynomials. Let
    \[
    K=\{(x,y)\in\mathbb R^2 :\ f_i(x,y)\geq 0
    \text{ for all }i=1,\ldots,n\}.
    \]
    If $K=\varnothing$, then
    \[
    p_2\brac{R[x,y,\sqrt{f_1},\ldots,\sqrt{f_n}]} <\infty.
    \]
\end{prop}
\begin{proof}
    We have an isomorphism
    \[
    \R[x,y,\sqrt{f_1},\ldots,\sqrt{f_n}] \cong \R[x,y,z_1,\dots,z_n]/I,
    \]
    where
    \[
    I = \brac{z_1^2-f_1,\dots,z_n^2-f_n}.
    \]
    Then
    \[
    \mathcal V_\R(I) = \left\{ (x,y,z_1,\dots,z_n)\in\R^{n+2} : \ z_i^2=f_i(x,y)
    \text{ for all }i=1,\ldots,n \right\}.
    \]
    Note that if $K=\varnothing$, then $\mathcal{V}_\R(I)=\varnothing$. Thus, by the real Nullstellensatz, we get that $-1$ is a sum of squares of polynomials modulo $I$. Hence, passing to the quotient, we can write
    \[
    -1 = \sum_{i=1}^k g_i^2,\qquad\text{where }g_i\in \R[x,y,\sqrt{f_1},\ldots,\sqrt{f_n}].
    \]
    Thus, given an arbitrary element $a\in \R[x,y,\sqrt{f_1},\ldots,\sqrt{f_n}]$, we have
    \[
    a = \brac{\frac{a+1}{2}}^2 + \sum_{i=1}^k \brac{\frac{a-1}{2}\cdot g_i}^2,
    \]
    and so
    \[
    p_2\brac{R[x,y,\sqrt{f_1},\ldots,\sqrt{f_n}]} \leq k+1<\infty.
    \]
\end{proof}

\begin{remark}
    It is worth to point out the problem of computing 2-Pythagoras number for projective normal surfaces was carried out in \cite{osz2025}. Surfaces whose coordinate rings are considered in this paper can be smooth or singular. 
\end{remark}

\vspace{0.3cm}
We now turn our attention to general even powers. It turns out that changing the approach slightly, it is possible to extend the result from Theorem~\ref{positive-leading} to higher Pythagoras numbers. 

\begin{thm}
    Let $f_1,\dots, f_n\in \R[x,y]$ be square-independent, non-constant polynomials. Assume that their $y$-leading coefficients are positive. Then for any integer $s>1$ we have
    \[
    p_{2s}\brac{\R[x,y,\sqrt{f_1},\dots, \sqrt{f_n}]} = \infty.
    \]
\end{thm}

\begin{proof}
    We will carry out the proof for the case $s=2$. The argument for the higher powers works similarly.

    Similarly to Theorem~\ref{positive-leading}, consider a sequence of positive integers $(r_m)_{m\geq1}$ such that
    \[
    r_1>1+\max_i \deg f_i\quad \text{and} \quad r_k>\sum_{i=1}^{k-1} r_i. 
    \]
    Further, define a sequence of polynomials $(F_m)_{m\geq 1}$ by letting $F_1=1$ and
    \[
    F_{n+1} = F_n\brac{y-x^{r_n}}^4+1.
    \]

    Assume now to the contrary that
    \[
     p_{2s}\brac{\R[x,y,\sqrt{f_1},\dots, \sqrt{f_n}]} = L<\infty.
    \]
    Then we have
    \begin{equation}
     F_{L+1} = \sum_{i=1}^L h_i^4,
        \label{fourth-power-basic}
    \end{equation}
    for some $h_i\in \R[x,y,\sqrt{f_1},\dots, \sqrt{f_n}]$. Write
    \[
    h_i=\sum_{S\subseteq[n]}a_{i,S}\sqrt{f_S}.
    \]

    Consider now the group of automorphisms
    \[
    G = \operatorname{Aut}_{\mathbb R[x,y]}\brac{\R[x,y,\sqrt{f_1},\dots, \sqrt{f_n}]}
    \]
    of $\R[x,y,\sqrt{f_1},\dots, \sqrt{f_n}]$ as an $\R[x,y]$ algebra. 
    Note that every element of $G$ is determined by its value at $\sqrt{f_i}$ which is sent to a root of $T^2-f_i$, hence to $\pm\sqrt{f_i}$. Conversely, by square-independence, the independent sign changes of the elements $\sqrt{f_i}$ all define automorphisms. Thus $G$ is precisely the group of sign changes of $\sqrt{f_i}$. In particular $G\cong(\mathbb Z/2\mathbb Z)^n$.
    
    Applying any element $\sigma\in G$ to equation \eqref{fourth-power-basic} yields
    \[
    F_{L+1} = \sum_{i=1}^L \sigma(h_i)^4.
    \]
    For any $u\in\R[x,y,\sqrt{f_1},\dots, \sqrt{f_n}]$ denote for convenience
    \[
    \overline{u} := u(x,x^{r_L}),
    \]
    and write $\gamma = y-x^{r_L}$. Setting $\gamma=0$ gives us
    \[
    1 = F_{L+1}(x,x^{r_L}) = \sum_{i=1}^L\overline{\sigma(h_i)}^{\, 4}.
    \]
    It follows that every function $\overline{\sigma(h_i)}$ is bounded for sufficiently large $x$.

    For $T\subseteq[n]$, let 
    \[
    \chi_T:G\longrightarrow\{\pm1\}
    \]
    be the character determined by
    \[
    \sigma(\sqrt{f_T}) = \chi_T(\sigma)\sqrt{f_T}.
    \]
    
    Fix a subset $T\subseteq[n]$, an index $i$ and consider a weighted average of the form
    \begin{equation}
    \frac{1}{|G|}\sum_{\sigma\in G} \chi_T(\sigma)\sigma(h_i) = \frac{1}{|G|}\sum_{\sigma\in G} \chi_T(\sigma)\sum_{S\subseteq[n]}a_{i,S}\chi_S(\sigma)\sqrt{f_S}.
        \label{character-average}
    \end{equation}
    Recall, that by character orthogonality we have (see \cite{kowalski2021} for details)
    \[
    \frac{1}{|G|} \sum_{\sigma\in G} \chi_T(\sigma)\chi_S(\sigma) = 
    \begin{cases}
    1,& \text{if}\;S=T \\
    0,& \text{if}\;S\neq T
    \end{cases}
    \]
    Hence we can rearrange \eqref{character-average} into
    \[
    \sum_{S\subseteq[n]}a_{i,S}\sqrt{f_S}\brac{\frac{1}{|G|} \sum_{\sigma\in G} \chi_T(\sigma)\chi_S(\sigma)} = a_{i,T}\sqrt{f_T}.
    \]
    After substituting $\gamma=0$ this average thus yields
    \[
    \frac{1}{|G|}\sum_{\sigma\in G} \chi_T(\sigma)\overline{\sigma(h_i)} = a_{i,T}(x,x^{r_L}) \sqrt{f_T(x,x^{r_L})}.
    \]
    Since $\overline{\sigma (h_i)}$ are bounded, the same must be true for $a_{i,T}(x,x^{r_L}) \sqrt{f_T(x,x^{r_L})}$. Hence
    \[
        a_{i,T}(x,x^{r_L})=0 \qquad(\varnothing\neq T\subseteq[n]),
    \]
    and
    \[
            a_{i,\varnothing}(x,x^{r_L})=c_i\in\mathbb R,\qquad \sum_{i=1}^L c_i^4 =1 .
    \]
    Consequently, we may write
    \[
    h_i = c_i+\gamma g_i,\qquad g_i\in \R[x,y,\sqrt{f_1},\dots, \sqrt{f_n}].
    \]

    Consider as before $D = 1+\max\deg f_i$, define a degree on monomials via
    \[
    \deg_D(x^ay^b) = a+Db,
    \]
    and extend it to $\R[x,y]$ in the natural way. We will now construct a degree function on $\R[x,y,\sqrt{f_1},\dots, \sqrt{f_n}]$ in the following way: for
    \[
    u = \sum_{S\subseteq[n]}k_{S}\sqrt{f_S},
    \]
    we write
    \[
    \deg_A u = \max_{S\subseteq[n]\,:\, k_S\neq0}\brac{ \deg_D k_S + \frac{1}{2}\deg_D f_S}.
    \]
    We have
    \[
    \deg_A(u+v)\leq \max(\deg_A u,\deg_A v)\quad\text{and }\quad \deg_A(uv)\leq \deg_A u+\deg_A v  .
    \]
    Furthermore
    \[
    \deg_A(pu) = \deg_D(p) + \deg_A(u),\qquad\text{for nonzero }\; p\in\R[x,y].
    \]
    Indeed, if
    \[
    u =  \sum_{S\subseteq[n]} k_S\sqrt{f_S},\qquad\text{then }\; pu = \sum_{S\subseteq[n]} (pk_S)\sqrt{f_S},
    \]
    and therefore
    \[
    \deg_A(pu)=\max_S(\deg_Dp+\deg_D k_S+\frac12 \deg_D f_S)=\deg_D p+\deg_A u.
    \]
    In particular $\deg_A(\gamma u) = r_L + \deg_A u$.

    Denote by
    \[
    \pi:\R[x,y,\sqrt{f_1},\ldots,\sqrt{f_n}]\longrightarrow\R[x,y]
    \]
    the projection onto the $f_\varnothing$ component. Let $u\in \R[x,y,\sqrt{f_1},\ldots,\sqrt{f_n}]$ be arbitrary. Then
    \[
    \pi(u^2) = \sum_{S\subseteq[n]}f_S k_{S}^2,\qquad\text{where }\; u=\sum_{S\subseteq[n]} k_S\sqrt{f_S}.
    \]
    Thus
    \[
    \deg_D \pi(u^2) = \max_S(2\deg_D k_{S}+ \deg_D f_S) = 2\deg_A u.
    \]
    On the other hand
    \[
    \deg_D \pi(u^2)\leq \deg_A u^2\leq 2\deg_A u,
    \]
    hence
    \[
    \deg_A u^2 = \deg_D \pi(u^2) = 2\deg_A u,
    \]
    for arbitrary $u$. It follows that
    \[
    \deg_D \pi(h_i^4) = 4\deg_A h_i.
    \]
    Moreover,
    \[
    F_{L+1} = \sum_{i=1}^L\pi(h_i^4),
    \]
    and the leading terms of the nonzero summands on the right cannot cancel. Therefore
    \[
    \deg_DF_{L+1} = 4\max_i\deg_Ah_i.
    \]
    Since
    \[
    \deg_DF_{L+1}=4\sum_{j=1}^Lr_j,
    \]
    we may conclude that
    \[
    \deg_Ah_i\leq\sum_{j=1}^Lr_j.
    \]
    On the other hand, $h_i=c_i+\gamma q_i$, and hence, whenever $q_i\neq 0$,
    \[
            \deg_Ah_i  =  \deg_A(\gamma q_i)   =   r_L+\deg_Aq_i.
    \]
    Consequently,
    \begin{equation}
        \deg_A q_i \leq r_1+\dots+r_{L-1}<r_L.
            \label{q-degree-bound}
    \end{equation}
    
    Recall that we have
    \[
    F_{L+1}=1+\gamma^4F_L = \sum_{i=1}^L(c_i+\gamma q_i)^4\qquad\text{and }
    \qquad\sum_{i=1}^L c_i^4=1.
    \]
    After cancelling the $1$'s we are left with
    \begin{equation}
            \gamma^4F_L =
        4\gamma\sum_i c_i^3q_i
        +6\gamma^2\sum_i c_i^2q_i^2
        +4\gamma^3\sum_i c_iq_i^3
        +\gamma^4\sum_i q_i^4.
            \label{main-fourth-equation}
    \end{equation}

    Dividing by $\gamma$ and reducing modulo $\gamma$ shows that
    \[
            \gamma\mid \sum_i c_i^3q_i.
    \]
    However,
    \[
    \deg_A\brac{ \sum_i c_i^3q_i } \leq \max_i\deg_A q_i<r_L
    \]
    and every nonzero element divisible by $\gamma$ has degree at least $\deg_D\gamma = r_L$. Therefore the divisibility is only possible if
    \[
    \sum_i c_i^3 q_i = 0.
    \]
    Now, dividing \eqref{main-fourth-equation} by $\gamma^2$ and taking modulo $\gamma$, we obtain
    \begin{equation}
    \sum_i c_i^2q_i^2\in\gamma \R[x,y,\sqrt{f_1},\ldots,\sqrt{f_n}].
        \label{modulo-gamma}
    \end{equation}
    Write
    \[
    q_i=\sum_{S\subseteq[n]}b_{i,S}\sqrt{f_S}.
    \]
    Note that $\pi$ is $\R[x,y]$-linear, so applying it to \eqref{modulo-gamma} yields
    \[
    \sum_{i=1}^L\sum_{S\subseteq[n]} c_i^2 f_S b_{i,S}^2 \in \gamma\mathbb R[x,y].
    \]
    Substituting $\gamma=0$ leaves us with
    \[
    \sum_{i=1}^L\sum_{S\subseteq[n]} c_i^2 f_S(x,x^{r_L}) b_{i,S}(x,x^{r_L})^2 = 0.
    \]
    Every nonzero summand here has positive leading coefficient. Therefore, whenever $c_i\neq 0$ we must have
    \[
    b_{i,S}(x,x^{r_L})=0 \qquad\text{for every }S\subseteq[n].
    \]
    Thus
    \[
    \gamma\mid q_i
    \qquad(c_i\neq0),
    \]
    but from \eqref{q-degree-bound} we know that $\deg_A q_i<r_L$, whereas every nonzero multiple of $\gamma$ has degree at least $r_L$. Thus $q_i=0$ whenever $c_i$ is nonzero. Hence
    \[
    h_i=
    \begin{cases}
    c_i, & c_i\neq0,\\
    \gamma q_i, & c_i=0.
    \end{cases}
    \]

    Returning to \eqref{main-fourth-equation} we now have (after dividing by $\gamma^4$)
    \[
        F_L=\sum_{c_i=0}q_i^4.
    \]
    Since we know that at least one $c_i$ is nonzero, this represents $F_L$ as a sum of at most $L-1$ fourth powers. Repeating the same argument with $F_L,F_{L-1},\ldots,F_2$, we eventually obtain a representation of $F_1=1$ as a sum of zero fourth powers, which is impossible. This contradiction proves that
    \[
    p_4\brac{\R[x,y,\sqrt{f_1},\ldots,\sqrt{f_n}]}=\infty.
    \]

\end{proof}

\section{Sums of even powers of $0$-regulous functions.}
In this section we will discuss sums of even powers of 0-regulous functions. 

The third author and Vill proposed the following conjecture.
\begin{conj}\cite[Conjecture 5.2]{kv2023}
Let $R$ be a commutative ring with identity. Then the following conditions are equivalent.
\begin{enumerate}
\item $p_2(R)<\infty$
\item $p_{2n}(R) <\infty$ for some $n$
\item $p_{2n}(R) <\infty$ for all $n$.
\end{enumerate}
\end{conj}

The above conjecture is known to be true for fields (\cite{becker1982}) and some polynomial rings \cite{bk2024, kv2023}. Also, it is true for valuation rings (this follows from unital Pythagoras number \cite[Theorem 5.17]{bp1996}). Here, we will prove it for the ring of $0$-regulous functions on an irreducible $0$-regulous variety.  Let us provide necessary definitions.
\begin{df}
Let $s$ be a positive integer and $k$ be a nonnegative integer. We say that a continuous function $f: \R^s \rightarrow \R$ is $k$-regulous on $\R^s$ if $f$ is of class $\mathcal{C}^k$ and $f$ is a rational function, i.e. there exists a non-empty Zariski open subset $U\subset \R^s$ such that $f|_U$ is regular.
\end{df}
Family of $k$-regulous functions defines a topology on $\R^s$ which is called constructible topology. If $X\subset \R^s$ is a closed $k$-regulous subset, we define the ring of $k$-regulous functions on $X$ as a quotient ring 
$$\RR^k(X):=\RR^k(\R^s)/\mathcal{I}(X)$$
where $\mathcal{I}(X)$ is the set of all $k$-regulous functions on $\R^s$ which vanish on $X$.

For basic properties of the ring of $k$-regulous functions on a $k$-regulous variety we refer the reader to \cite{bk2023, fhmm2016} (see also \cite{banecki2024, banecki2025} for some recent results). Main theorem of this section is the following.
\begin{thm}
    Let $X$ be an irreducible $0$-regulous variety of dimension $s$. Then 
    $$ p_{2n}(\RR^0(X)) < \infty $$
    for any positive integer $n$.
\end{thm}
Before we prove the above, we need some preparations. Let $R$ be a commutative ring with identity. 

\begin{df}
We define the $2n$-th unital Pythagoras number of $R$,  $p^*_{2n}(R)$, to be the least positive integer $\ell$ such that any sum of $2n$-th powers of elements of $A$, which is a unit, can be expressed as a sum of $\ell$ $2n$-th powers of elements of $A$. If such a number does not exist, we put $p^*_{2n}(R)=\infty$.
\end{df}
Of course, for any ring $R$ and any positive integer $n$  we have an inequality $p^*_{2n}(R) \leq p_{2n}(R)$. As noted in \cite{bp1996}, one of this numbers can be finite while the other is not. For $R=\Z[x]$ we have $p^*_{2n}(\Z[x])=1$ and $p_{2n}(\Z[x])=\infty$ (cf. \cite{cldr1982}).

Becker and Powers studied the above invariant.

\begin{thm}\label{Becker Powers}\cite[Theorem 5.17]{bp1996}
Let $R$ be a commutative ring with identity. Assume that $1+\sum R^2 \subset R^*$. Then the following conditions are equivalent.
\begin{enumerate}
\item $p^*_2(R)<\infty$
\item $p^*_{2n}(R) <\infty$ for some $n$
\item $p^*_{2n}(R) <\infty$ for all $n$.
\end{enumerate}
Additionally, if $p^*_2(R)=p < \infty $ then
$$p_{2n}^*(R) \leq p^2 {2n+2+p \choose p}G(2n).$$
Here, $G(n)$ is the function related to the classical Waring problem.
\end{thm}

Let us recall one more result

\begin{lemma}\label{Banecki lemat przedłuzanie}\cite[Lemma 2.8]{bk2023}
Let $X\subset\R^n$ be an affine 0-regulous variety and $f,g\in\RR^0(X)$ be $0$-regulous functions such that
\begin{equation*}
    Z=\{x\in X:f(x)=0\}\subset\{x\in X:g(x)=0\}.
\end{equation*}
If the function $h:X \rightarrow \R$ defined as
\begin{equation*}
    h:=
    \begin{cases}
        \frac{g}{f} \text{ on } X \backslash Z \\
        0 \text{ otherwise}
    \end{cases}
\end{equation*}
\end{lemma}
is continuous, then $h \in\RR^0(X)$.

\noindent
\textit{Proof of Theorem 3.3}
    If $n=1$ then  \cite[Theorem 2.6]{bk2023} gives us $p_2(\RR^0(X)) \leq 2^s$. Also, let us stress out that if $V \subset X$ is an open constructible set, then $V$ is again a $0$-regulous variety of dimension $s$ \cite[Proposition 1.3]{bk2023}, hence $p_2(\RR^0(V)) \leq 2^s$.

Assume now $n>1$ and let $f \in \RR^0(X)$ be a sum of $2n$th powers. Denote by $Z=\mathcal{Z}(F)$ - the zero set of $f$, and let $U = X\setminus Z$ be its complement.
It follows from the case $n=1$ and Theorem \ref{Becker Powers} that there exists a positive integer $N_n$ such that $p_{2n}^*(\RR^0(V)) \leq N_n$ for any $s$-dimensional $0$-regulous irreducible set.

Consider now $f \in \RR^0(U)$. It is a strictly positive function on $U$, hence it is a unit. By the above discussion, we may write $f=\sum_{i=1}^{N_n}f_i^{2n}$, where $f_i \in \RR^0(U)$. We know that $\RR^0(X)_f=\RR^0(U)$ by \cite[Proposition 1.4]{bk2023}, and so, we may write $f_i=\frac{g_i}{f^{l_i}}$ for some $g_i \in \RR^0(X)$ and nonnegative integers $l_i$, for $i=1,2,\dots, N_n$. Without loss of generality we may assume that all of the $l_i$'s are equal to some positive integer $N$.

We may now define functions 
\begin{equation*}
    h_i:=
    \begin{cases}
        \frac{g_i}{f^N} \text{ on } U \\
        0 \text{ on } Z
    \end{cases}
\end{equation*}
for $i=1,2,\dots, N_n$. Such functions are continuous on $X$, and are $0$-regulous by Lemma \ref{Banecki lemat przedłuzanie}. This finishes the proof.

\begin{remark}
    The above proof is in a similar spirit as the proof of \cite[Theorem 2.6]{bk2023}. The ring of $0$-regulous functions is the first known example (other than fields and valuation rings) for which we are able to show finiteness of the higher even Pythagoras numbers. The extendability of quotients of such functions is crucial. It is not known how to tackle similar problems for the ring of regular functions (cf. \cite[Remark 4.14]{kowalczyk2025}) or $k$-regulous functions for $k>0$.
\end{remark}
\section{Bad points of higher order}

In this chapter we will study size of \textit{the bad set of a polynomial}.

\begin{df}
Let $f \in \R[x_1,x_2,\dots, x_n]$ be a nonnegative polynomial. Define

        \begin{align*}
        \mathrm{Denom}_{2n}(f) &= \Bigl\{ G \in \mathbb{R}[x_1,\dots, x_s]  \mid G^{2n} \cdot f = \sum_{i =1}^N H_i^{2n} \text{ for some } H_i \in \R[x_1,\dots, x_s] \Bigr\}, \\   
       % \mathrm{Bad}_{2k}(f) &= \Bigl\{ x \in \mathbb{R}^n \mid \text{for all } G \in \mathrm{Denom}_{2k}(f) \text{ equality } G(x)=0 \text{ holds} \Bigr\}.
        \end{align*}
Let $I_{2n}(f)$ be the ideal generated by $\mathrm{Denom}_{2n}(f)$. We define the set of bad points of order $2n$ of $f$, $\mathrm{Bad}_{2n}(f)$ as the zero set of the ideal $I_{2n}(f)$.

\end{df}
Of course, it may happen that there does not exist such $G$ as above, i.e. $f$ is not a sum of $2n$th powers in the field of fractions, then $\mathrm{Denom}_{2n}(f)=\{0\}$. The set of the bad points of order $2n$ of $f$ is the common zero set of all possible denominators in all representations of $f$ a sum of $2n$th powers of rational functions.

Firstly, let us discuss what is known for $n=1$.

\begin{thm}
    Let $X \subset \R^s$ be a nonsingular irreducible algebraic surface and $f \in \R[X]$ be a polynomial function which is a sums of squares in $\R(X)$. Then $\mathrm{codim}(\mathrm{Bad}_2(f)) \geq 3$.
\end{thm}
If $X=\R^2$ then this is precisely \cite[Proposition 5.1]{delzellPhd}. The more general case follows from the result of Scheiderer \cite[Theorem 4.8]{scheiderer2001} that psd=sos holds for local regular rings of dimension 2. Bad points of order greater than 2 have not been studied before. 

Here, we are able to show the following result.

\begin{thm} \label{THM: Bad_2k może mieć codim 2 dla k>1}
   Let $n>1$ be a positive integer. Then there exists a polynomial $f \in \R[x,y]$ which is a sum of $2n$th powers of rational functions and $\mathrm{codim}\big(\mathrm{Bad}_{2n}(f)\big) = 2$ holds.
\end{thm}

The above theorem will be proved by constructing an explicit examples. Let us first recall crucial fact.

\begin{thm}\cite[Theorem 1.9]{becker1982}
    Polynomial $f \in \R[x]$ can be expressed as
    $
    f = \sum_{i=1}^N {\big(\frac{g_i}{h}\big)}^{2n}
    $
    for some $g_i \in \R[x]$ and $h \in \R[x] \backslash \{0\}$
    if and only if the following conditions are satisfied
    \begin{enumerate}
        \item $f(x) \geq 0$ for all $x \in \R$,
        \item $2n \mid \deg(f)$, and
        \item $2n$ divides the multiplicity of $x - x_0$ in $f$ for every $x_0 \in \R$.
    \end{enumerate}
\end{thm}
For the discussion of the above result, see also \cite[Introduction]{clpr1996}. Also, largest positive integer $N$ in the above theorem is the value of $p_{2n}(\R(x))$ which is studied in \cite{clpr1996}.

\noindent
\textit{Proof of Theorem \ref{THM: Bad_2k może mieć codim 2 dla k>1}:}
Obviously it cannot happen that  $\mathrm{codim}\big(\mathrm{Bad}_{2n}(f)\big) =0$. If this value is equal to 1 then in every presentation $$G^{2n}f =\sum_{i=1}^N f_i^{2n}$$ $G$ and $f_i$ for every $i=1,2,\dots, N$ vanish on a codimension 1 real algebraic subset of the real plane. However, as the ring $\R[x,y]$ is factorial, ideal of the set of such subset if principal, hence we may factor them out. As a consequences, the codimension of the bad set cannot be 1.

Let us now discuss the case $n=2$, as higher powers will follow analogously. 

Consider family of polynomials given by 
$$f_m = x^4+mx^2+1.$$
$f_m$ is a sum of fourth powers of polynomials iff $m\in [0,6]$ (cf. 
\cite[Exercise 7.4.4]{pd2001}). By the previous theorem, $f_m$ is a sum of fourth powers of rational functions iff $m\geq 0.$ Hence $f_7$ is not a sum of fourth powers of polynomials, but there exists a a strictly positive polynomial $h \in \R[x]$ such that $h^4f_7$ is a sum of fourth powers of polynomials. 

Denote by $\widetilde{f_7}=x^4+7x^2y^2+y^4$ its homogenization. If $\widetilde{h}$ is a homogenization of $h$ then $\widetilde{h}^4\widetilde{f_7}$ is a sum of fourth powers of polynomials. Of course homogenization of a strictly positive polynomial is a positive definite form, hence $\widetilde{h}$ vanishes only at the origin. As a consequence, $\mathrm{Bad}_4(f_7) \subset \{(0,0) \}$. It is enough to show that this set is non empty. 

Assume by contrary that the bad set of order 4 is empty.
If there exists a polynomial $g \in \mathrm{Denom}_4(\widetilde{f}_7)$ such that $g(0,0)\neq 0$, then $g^{4}+\widetilde{h}^{4} \in \mathrm{Denom}_4(\widetilde{f}_7)$, and it does not vanish anywhere. Without loss of generality, we may assume that there is $g \in \mathrm{Denom}(\widetilde{f}_7)$ which does not vanish. We have an equality
\begin{equation}\label{rownanko}
    g^4\widetilde{f}_7=\sum_{i=1}^Nf_i^4.
\end{equation}
Consider now the above equation in the ring $\R[[x,y]]$. In this ring, $g$ is a unit, and we may obtain the following

$$\widetilde{f}_7 = \sum_{i=1}^N H_i^4$$
for some power series $H_i \in \R[[x,y]]$.
By comparing terms of degree four we get a contradiction, since $\widetilde{f}_7$ is not a sum of fourth powers of polynomials. This shows that the bad set of order $4$ of $f_7$ consists of the origin and finishes the proof for $n=2$.

For $n>2$ one can reason similarly taking the polynomial $f_n=x^{2n}+x^{2n-2}+1$. It is strictly positive and of degree $2n$ hence it is a sum of $2n$th powers of rational functions. Straightforward calculations shows that it is not a sum of $2n$th powers of polynomials.
This finishes the proof.
\hfill \qed

\section{Open problems}
In Section 3 we have discussed Beckers Theorem. Here we will propose a weaker version of the conjecture proposed in \cite{kv2023}.
\begin{prob}

Let $R$ be a commutative ring with identity, and $n>1$ be a positive integer.
Does the inequality 
$$ p_2(R)\leq p_{2n}(R)$$
 holds?

\end{prob}
The above is particularly interesting for rings such that $p_2(R)=\infty$. Every known computed example follows the above rule, however the values of $p_2(R)$ and $p_{2n}(R)$ are computed on their own, and usually by different methods. To the authors best knowledge, there is no known relation between the Pythagoras numbers for general rings.

\section*{Acknowledgment}
This paper was realized as a part of Research Projects of Mathematical Students Society initiative at Jagiellonian University.

\bibliographystyle{plain}
\bibliography{refs.bib}

@article{kv2023,
  author  = {Tomasz Kowalczyk and Julian Vill},
  title   = {On Higher {Pythagoras} numbers of real polynomial rings},
  journal = {Indiana Univ. Math. J.},
  year    = {2026},
  note    = {To appear}
}

@Article{bk2023,
 Author = {Banecki, Juliusz and Kowalczyk, Tomasz},
 Title = {Sums of even powers of {{\(k\)}}-regulous functions},
 FJournal = {Indagationes Mathematicae. New Series},
 Journal = {Indag. Math., New Ser.},
 ISSN = {0019-3577},
 Volume = {34},
 Number = {3},
 Pages = {477--487},
 Year = {2023},
 Language = {English},
 DOI = {10.1016/j.indag.2022.12.004},
 zbMATH = {7677394},
 Zbl = {1526.11011}
}

@Article{bp1996,
 Author = {Becker, Eberhard and Powers, Victoria},
 Title = {Sums of powers in rings and the real holomorphy ring},
 FJournal = {Journal f{\"u}r die Reine und Angewandte Mathematik},
 Journal = {J. Reine Angew. Math.},
 ISSN = {0075-4102},
 Volume = {480},
 Pages = {71--103},
 Year = {1996},
 Language = {English},
 DOI = {doi:10.1515/crll.1996.480.71},
 URL = {www.degruyter.com/document/doi/10.1515/crll.1996.480.71/pdf},
 zbMATH = {950784},
 Zbl = {0922.12003}
}

@Misc{becker1982,
 Author = {Becker, Eberhard},
 Title = {The real holomorphy ring and sums of {{\(2n\)}}-th powers},
 Year = {1982},
 Language = {English},
 HowPublished = {G{\'e}om{\'e}trie alg{\'e}brique r{\'e}elle et formes quadratiques, {Journ{\'e}es} {Soc}. {Math}. {Fr}., {Univ}. {Rennes} 1981, {Lect}. {Notes} {Math}. 959, 139-181.},
 zbMATH = {3801665},
 Zbl = {0508.12020}
}

@article{fhmm2016,
 author = {Fichou, Goulwen and Huisman, Johannes and Mangolte, Fr{\'e}d{\'e}ric and Monnier, Jean-Philippe},
 title = {Regulous functions},
 fjournal = {Journal f{\"u}r die Reine und Angewandte Mathematik},
 journal = {J. Reine Angew. Math.},
 issn = {0075-4102},
 volume = {718},
 pages = {103--151},
 year = {2016},
 language = {French},
 doi = {10.1515/crelle-2014-0034},
 zbMATH = {6626588},
 Zbl = {1390.14172}
}

@Article{bk2024,
 Author = {B{\l}achut, Kacper and Kowalczyk, Tomasz},
 Title = {Sums of squares on hypersurfaces},
 FJournal = {Results in Mathematics},
 Journal = {Result. Math.},
 ISSN = {1422-6383},
 Volume = {79},
 Number = {2},
 Pages = {11},
 Note = {Id/No 90},
 Year = {2024},
 Language = {English},
 DOI = {10.1007/s00025-023-02118-8},
 zbMATH = {7802386}
}

@article{banecki2025,
 author = {Banecki, Juliusz},
 title = {Extensions of {{\(k\)}}-regulous functions from two-dimensional varieties},
 fjournal = {Mathematische Annalen},
 journal = {Math. Ann.},
 issn = {0025-5831},
 volume = {391},
 number = {2},
 pages = {2541--2585},
 year = {2025},
 language = {English},
 doi = {10.1007/s00208-024-02981-y},
 zbMATH = {7974021}
}

@misc{banecki2024,
      title={Extension of $k$-regulous functions from varieties of arbitrary dimension}, 
      author={Juliusz Banecki},
      year={2024},
      eprint={2412.14412},
 howpublished = {Preprint, {arXiv}:2412.14412 [math.{AG}] (2024)},

      archivePrefix={arXiv},
      primaryClass={math.AG},
      url={https://arxiv.org/abs/2412.14412}, 
}

@Article{cldr1982,
 Author = {Choi, Man-Duen and Dai, Zong Duo and Lam, Tsin Yuen and Reznick, Bruce},
 Title = {The {Pythagoras} number of some affine algebras and local algebras},
 FJournal = {Journal f{\"u}r die Reine und Angewandte Mathematik},
 Journal = {J. Reine Angew. Math.},
 ISSN = {0075-4102},
 Volume = {336},
 Pages = {45--82},
 Year = {1982},
 Language = {English},
 DOI = {10.1515/crll.1982.336.45},
 zbMATH = {3786936},
 Zbl = {0499.12018}
}

@misc{kowalczyk2025,
      title={Sums of squares of regular functions on rational surfaces}, 
      author={Tomasz Kowalczyk},
      year={2025},
      eprint={2407.20378},
 howpublished = {Preprint, {arXiv}:2407.20378 [math.{AG}] (2025)},
      archivePrefix={arXiv},
      primaryClass={math.AG},
      url={https://arxiv.org/abs/2407.20378}, 
}

@Article{scheiderer2001,
 Author = {Scheiderer, Claus},
 Title = {On sums of squares in local rings},
 FJournal = {Journal f{\"u}r die Reine und Angewandte Mathematik},
 Journal = {J. Reine Angew. Math.},
 ISSN = {0075-4102},
 Volume = {540},
 Pages = {205--227},
 Year = {2001},
 Language = {English},
 DOI = {10.1515/crll.2001.085},
 URL = {semanticscholar.org/paper/2f68210c4fcc49d86c7593b508724a2c4e243513},
 zbMATH = {1670928},
 Zbl = {0991.13014}
}

@Article{clpr1996,
 Author = {Choi, Man-Duen and Lam, Tsin Yuen and Prestel, Alexander and Reznick, Bruce},
 Title = {Sums of {{\(2m\)}}-th powers of rational functions in one variable over real closed fields},
 FJournal = {Mathematische Zeitschrift},
 Journal = {Math. Z.},
 ISSN = {0025-5874},
 Volume = {221},
 Number = {1},
 Pages = {93--112},
 Year = {1996},
 Language = {English},
 DOI = {10.1007/PL00004512},
 zbMATH = {861937},
 Zbl = {0859.12002}
}

@misc{osz2025,
      title={Finiteness of {Pythagoras} numbers of finitely generated real algebras}, 
      author={Yi Ouyang and Qimin Song and Chenhao Zhang},
      year={2025},
      eprint={2503.17751},
      archivePrefix={arXiv},
      primaryClass={math.NT},
      url={https://arxiv.org/abs/2503.17751}, 
}

@article{scheiderer2017,
 author = {Scheiderer, Claus},
 title = {Sum of squares length of real forms},
 fjournal = {Mathematische Zeitschrift},
 journal = {Math. Z.},
 issn = {0025-5874},
 volume = {286},
 number = {1-2},
 pages = {559--570},
 year = {2017},
 language = {English},
 doi = {10.1007/s00209-016-1773-z},
 zbMATH = {6726613},
 Zbl = {1420.11070}
}

@phdthesis{delzellPhd,
  author  = "Delzell, Charles Neal",
  title   = "A constructive, continuous solution to Hilbert 17th problem, and other results in semialgebraic geometry,",
  school  = "Stanford University",
  year    = "1980"
}

@Book{pd2001,
 Author = {Prestel, Alexander and Delzell, Charles N.},
 Title = {Positive polynomials. {From} {Hilbert}'s 17th problem to real algebra},
 FSeries = {Springer Monographs in Mathematics},
 Series = {Springer Monogr. Math.},
 ISSN = {1439-7382},
 ISBN = {3-540-41215-8},
 Year = {2001},
 Publisher = {Berlin: Springer},
 Language = {English},
 zbMATH = {1601019},
 Zbl = {0987.13016}
}

@article{frs2004,
 author = {Fernando, Jos{\'e} F. and Ruiz, Jes{\'u}s M. and Scheiderer, Claus},
 title = {Sums of squares in real rings},
 fjournal = {Transactions of the American Mathematical Society},
 journal = {Trans. Am. Math. Soc.},
 issn = {0002-9947},
 volume = {356},
 number = {7},
 pages = {2663--2684},
 year = {2004},
 language = {English},
 doi = {10.1090/S0002-9947-03-03438-X},
 zbMATH = {2091175},
 Zbl = {1080.14071}
}

@article{km2024,
 author = {Kowalczyk, Tomasz and Miska, Piotr},
 title = {On {Waring} numbers of henselian rings},
 fjournal = {Mathematika},
 journal = {Mathematika},
 issn = {0025-5793},
 volume = {70},
 number = {4},
 pages = {28},
 note = {Id/No e12276},
 year = {2024},
 language = {English},
 doi = {10.1112/mtk.12276},
 zbMATH = {8015217}
}

@Article{grimm2015,
 Author = {Grimm, David},
 Title = {Lower bounds for {Pythagoras} numbers of function fields},
 FJournal = {Commentarii Mathematici Helvetici},
 Journal = {Comment. Math. Helv.},
 ISSN = {0010-2571},
 Volume = {90},
 Number = {2},
 Pages = {365--375},
 Year = {2015},
 Language = {English},
 DOI = {10.4171/CMH/356},
 zbMATH = {6451262},
 Zbl = {1334.12002}
}

@misc{dghmy2026,
 author = {Nicolas Daans and Stevan Gajovi{\'c} and Siu Hang Man and Pavlo Yatsyna},
 title = {Pythagoras numbers for infinite algebraic fields},
 year = {2026},
 howpublished = {Preprint, {arXiv}:2502.11222 [math.{NT}] (2026)},
 doi = {10.1007/s13163-026-00562-y},
 url = {https://arxiv.org/abs/2502.11222},
 arXiv = {arXiv:2502.11222}
}

@article{santiago2025,
 author = {Laplagne, Santiago},
 title = {On the {Pythagoras} number for polynomials of degree 4 in 5 variables},
 fjournal = {Revista de la Uni{\'o}n Matem{\'a}tica Argentina},
 journal = {Rev. Uni{\'o}n Mat. Argent.},
 issn = {0041-6932},
 volume = {68},
 number = {1},
 pages = {343--348},
 year = {2025},
 language = {English},
 doi = {10.33044/revuma.4224},
 zbMATH = {8103035},
 Zbl = {1573.14179}
}

@article{tinkova2025,
 author = {Tinkov{\'a}, Magdal{\'e}na},
 title = {Bounds on the {Pythagoras} number and indecomposables in biquadratic fields},
 fjournal = {Proceedings of the Edinburgh Mathematical Society. Series II},
 journal = {Proc. Edinb. Math. Soc., II. Ser.},
 issn = {0013-0915},
 volume = {68},
 number = {3},
 pages = {843--868},
 year = {2025},
 language = {English},
 doi = {10.1017/S0013091525000112},
 zbMATH = {8086103}
}

@article{bds2025,
 author = {Blekherman, Grigoriy and Dunbar, Alex and Sinn, Rainer},
 title = {Pythagoras numbers for ternary forms},
 fjournal = {Proceedings of the American Mathematical Society},
 journal = {Proc. Am. Math. Soc.},
 issn = {0002-9939},
 volume = {153},
 number = {10},
 pages = {4177--4195},
 year = {2025},
 language = {English},
 doi = {10.1090/proc/17336},
 zbMATH = {8085817}
}

@misc{benoist2025,
 author = {Olivier Benoist},
 title = {The {Pythagoras} number of fields of transcendence degree {$1$} over {$\mathbb{Q}$}},
 year = {2025},
 howpublished = {Preprint, {arXiv}:2506.21380 [math.{AG}] (2025)},
 url = {https://arxiv.org/abs/2506.21380},
 arXiv = {arXiv:2506.21380}
}

@misc{dombek2025,
 author = {Daniel Dombek},
 title = {On biquadratic fields: when 5 squares are not enough},
 year = {2025},
 howpublished = {Preprint, {arXiv}:2506.20820 [math.{NT}] (2025)},
 doi = {10.1007/s10998-026-00720-1},
 url = {https://arxiv.org/abs/2506.20820},
 arXiv = {arXiv:2506.20820}
}

@misc{hjp2025,
 author = {Jong In Han and Jaewoo Jung and Euisung Park},
 title = {On quadratic persistence and {Pythagoras} numbers of totally real projective varieties},
 year = {2025},
 howpublished = {Preprint, {arXiv}:2506.13247 [math.{AG}] (2025)},
 url = {https://arxiv.org/abs/2506.13247},
 arXiv = {arXiv:2506.13247}
}

@misc{kowalski2021,
 author = {Kowalski, Emmanuel},
 title = {Exponential sums over finite fields: elementary methods},
 year = {2021},
 howpublished = {Lecture notes, ETH Z{\"u}rich, version of September 14, 2021},
 url = {https://people.math.ethz.ch/~kowalski/exponential-sums-elementary.pdf}
}

@article{krs2022,
 author = {Kr{\'a}sensk{\'y}, Jakub and Ra{\v{s}}ka, Martin and Sgallov{\'a}, Ester},
 title = {Pythagoras numbers of orders in biquadratic fields},
 fjournal = {Expositiones Mathematicae},
 journal = {Expo. Math.},
 issn = {0723-0869},
 volume = {40},
 number = {4},
 pages = {1181--1228},
 year = {2022},
 language = {English},
 doi = {10.1016/j.exmath.2022.06.002},
 zbMATH = {7629342},
 Zbl = {1508.11047}
}

\begin{small}

\vspace{5pt}

\noindent
Bartłomiej Bychawski

\noindent
Institute of Mathematics

\noindent
Faculty of Mathematics and Computer Science

\noindent
Jagiellonian University

\noindent
ul. Łojasiewicza 6, 30-348 Kraków, Poland

\noindent
e-mail: bartlomiej.bychawski@student.uj.edu.pl

\vspace{5pt}

\noindent
Bartosz Głowacki

\noindent
Institute of Mathematics

\noindent
Faculty of Mathematics and Computer Science

\noindent
Jagiellonian University

\noindent
ul. Łojasiewicza 6, 30-348 Kraków, Poland

\noindent
e-mail: bartosz16.glowacki@student.uj.edu.pl

\vspace{5pt}

\noindent
Tomasz Kowalczyk

\noindent
Institute of Mathematics

\noindent
Faculty of Mathematics and Computer Science

\noindent
Jagiellonian University

\noindent
ul. Łojasiewicza 6, 30-348 Kraków, Poland

\noindent
e-mail: tomek.kowalczyk@uj.edu.pl

\end{small}

\end{document}